\documentclass{amsart}
\usepackage{amssymb}
\usepackage{xypic}
\usepackage[latin1]{inputenc}
\usepackage{graphics}

\def\N{{\mathbb N}}
\newcommand\casos[4]{
    \left\{
    \begin{array}{ll}
        #1 & \mbox{ #2 }\\
        #3 & \mbox{ #4 }\\
    \end{array}
    \right.
    }
\def\SS{{\mathcal S}}
\def\m{{\rm m}}
\def\Ap{{\rm Ap}}
\def\g{{\rm F}}

\newtheorem{theorem}{Theorem}
\newtheorem{lemma}[theorem]{Lemma}
\newtheorem{corollary}[theorem]{Corollary}
\newtheorem{proposition}[theorem]{Proposition}

\theoremstyle{remark}
\newtheorem{example}[theorem]{Example}
\newtheorem{remark}[theorem]{Remark}

\title{Patterns on numerical semigroups}
\author{Maria Bras-Amor\'os}
\thanks{
The first author was supported in part by the Spanish CICYT under Grant
TIC2003-08604-C04-01 and FEDER, Catalan DURSI under SGR2005-00319. The second
was supported by the project MTM2004-01446 and FEDER funds. The authors would
like to thank the referee for his/her many comments and suggestions.}

\address{Departament d'Enginyeria de la Informaci\'o i de les Comunicacions,
Universitat Aut\`onoma de Barcelona, E-08193 Bellaterra, Spain} 
\email{maria.bras@uab.es}

\author{Pedro A. Garc\'ia-S\'anchez}
\address{Departamento de \'Algebra, Universidad de Granada, E-18071 Granada, Spain}
\email{pedro@ugr.es}
\date{\today}


\begin{document}

\begin{abstract}
We introduce the notion of pattern for numerical semigroups, which allows us
to generalize the definition of Arf numerical semigroups. In this way
infinitely many other classes of numerical semigroups are defined giving a
classification of the whole set of numerical semigroups. In particular, all
semigroups can be arranged in an infinite non-stabilizing ascending chain
whose first step consists just of the trivial semigroup and whose second step
is the well known class of Arf semigroups. We describe a procedure to compute
the closure of a numerical semigroup with respect to a pattern. By using the
concept of system of generators associated to a pattern we construct
recursively a directed acyclic graph with all the semigroups admitting the
pattern.

\end{abstract}

\maketitle
\begin{center} \resizebox{\textwidth}{!}{Published in {\bf Linear Algebra and its Applications, Elsevier, vol. 414, pp. 652-669, April 2006.}} \end{center}

\section*{Introduction}

A {\em numerical semigroup} is a subset of $\N$ containing $0$, closed under addition and with
finite complement in $\N$ (here $\N$ denotes the set of nonnegative integers).
The theory of numerical semigroups is intimately related to the study of the non-negative integer
solutions of a linear equation in several unknowns with coefficients in $\N$
\cite{sylvester,brauer,brauer-shockley,johnson,selmer}.
Applications of numerical semigroups are found in the study of the
parameters of algebraic-geometry codes \cite{FeRa:dFR,KiPe:telescopic,HoLiPe:agc}.

For a numerical
semigroup $\Lambda$, the {\em multiplicity} of $\Lambda$, denoted by $\m(\Lambda)$, is the
smallest non-negative element of $\Lambda$, and the {\em conductor} of $\Lambda$ is the only
integer $c\in\Lambda$ such that $c-1\not\in\Lambda$ and $c+\N\subseteq\Lambda$
\cite{HoLiPe:agc}. Usually the element $c-1$ is known as the Frobenius number of $\Lambda$,
denoted here by $\g(\Lambda)$. Clearly, $\g(\Lambda)$ is the maximum of $\N\setminus \Lambda$.
Let $A$ be a subset of $\N$. The submonoid of $\N$ {\em generated} by $A\subseteq \N$ is the
smallest (with respect to set inclusion) submonoid of $\N$ containing $A$, and it is denoted
usually by $\langle A\rangle$, that is,
\[ \langle A\rangle =\left\{\sum_{i=1}^n k_i a_i ~|~ n\in \N, k_i\in \N, a_i\in A \hbox{ for all }
i\in \{1,\ldots,n\}\right\}.\] It is not hard to prove that $\langle A\rangle$ is a numerical
semigroup if and only if the greatest common divisor of the elements of $A$ is one. If $\Lambda$
is a numerical semigroup and $A$ is a subset of $\Lambda$, then we say that $A$ is a {\em system
of generators} (or simply that $A$ {\em generates} $\Lambda$) if $\langle A\rangle = \Lambda$.
We say that $A$ is a minimal system of generators of $\Lambda$ if in addition no proper subset
of $A$ generates $\Lambda$. Every numerical semigroup has a unique minimal system of generators.

A numerical semigroup $\Lambda$ is said to be {\em Arf} if for every $x,y,z\in
\Lambda$ with $x\geq y\geq z$, it holds that $x+y-z\in\Lambda$. Arf numerical
semigroups and their applications to coding theory have been widely studied in
\cite{BaDoFo,Bras:AAECC,BrOS,Bras:acute,CaFaMu:arf,RoGaGaBr}. In this work we
try to generalize the idea of Arf numerical semigroup to a semigroup
satisfying the condition that a certain expression on any decreasing sequence
of elements of the semigroup belongs always to the semigroup. The expression
is what we call a pattern.

Furthermore, in \cite{RoGaGaBr} the authors introduce the notion of the Arf
closure of a numerical semigroup $\Lambda$ as the smallest Arf numerical
semigroup containing $\Lambda$ (the idea of Arf closure appears for algebraic
curves in \cite{Arf}, though of course not with this name). In this work this
idea is generalized for other patterns and we define a procedure to obtain
such closure. We also extend the concept of an Arf-system of generators to any
pattern and show how to construct recursively a directed acyclic graph with
all the numerical semigroups admitting a given pattern.

In Section~\ref{section:patterns} we give the definition and some examples of patterns.
In Section~\ref{section:admissible} we characterize those patterns that can be admitted at least
by one numerical semigroup. This enables us to define admissible patterns.
Section~\ref{section:lock-relock} introduces the concept of strongly admissible pattern. The
advantage of these patterns is that one can effectively (computationally) deal with them. In
Section~\ref{section:closures} we give the definition and a procedure to obtain the closure of a
numerical semigroup with respect to a pattern. In the next section we introduce the concept of
$p$-system of generators for a numerical semigroup admitting the pattern $p$. The uniqueness of
minimal $p$-systems of generators for a given semigroup can be ensured when the pattern $p$ is
strongly admissible. We will show how to use this information to construct the set of all
numerical semigroups that admit a given strongly admissible pattern. In
Section~\ref{section:substraction} we show that each numerical semigroup admits infinitely many
patterns. In particular there will be a pattern giving information on ``how far'' from
substraction a semigroup is. This will yield an infinite non-stabilizing ascending chain of sets
of numerical semigroups containing all numerical semigroups. The concept of substraction pattern
generalizes that of the Arf pattern. In the last section we go one step beyond by presenting the
concept of boolean pattern, for which we can give invariants for equivalent patterns in this
class.

\section{Patterns}
\label{section:patterns}

A {\em pattern} $p$ of length $n$ is a linear homogeneous polynomial with non-zero integer
coefficients in $x_1,\dots,x_n$ (for $n=0$ the unique pattern is $p=0$). We say that a numerical
semigroup $\Lambda$ {\em admits} a pattern $p(x_1,\dots,x_n)$ if for every $n$ elements
$s_1,\dots, s_n$ in $\Lambda$ with $s_1\geq s_2\geq\dots\geq s_n$, the integer
$p(s_1,\dots,s_n)$ belongs to $\Lambda$. We denote by $\SS(p)$ the set of all numerical
semigroups admitting $p$.

\begin{example}
\label{ejemplo 1} Patterns with positive coefficients are admitted by all
numerical semigroups. The same trivially stands for the zero pattern.\qed
\end{example}

\begin{example}
Consider the pattern $p=x_1+x_2-x_3$. A numerical semigroup is Arf if and only
if it admits the pattern $p$. The pattern $p$ will be called the {\em Arf
pattern}.

A numerical semigroup is said to be of {\em maximal embedding dimension} if its multiplicity
equals the cardinality of its minimal system of generators (known as the embedding dimension of
the semigroup). From \cite[Proposition I.2.9]{BaDoFo}, one can easily derive that a numerical
semigroup $\Lambda$ is of maximal embedding dimension if and only if for all $x\geq y\in
\Lambda$, $x,y\geq \m(\Lambda)$, one has that $x+y-\m(\Lambda)\in \Lambda$. Observe that
$p(x_1,x_2)=x_1+x_2-\m(\Lambda)$ is not linear, and thus it is not a pattern. Note also, that
this in particular means that every Arf numerical semigroup is of maximal embedding dimension.
\qed
\end{example}

\begin{example}
The only numerical semigroup that admits the pattern $q=x_1-x_2$ is $\N$.
Indeed, suppose that $\Lambda$ admits this pattern and let $c$ be the
conductor of $\Lambda$. Consider $s_1=c+1$ and $s_2=c$. Since $\Lambda$ admits
$q$, $s_1-s_2=1$ belongs to $\Lambda$ and thus, $\Lambda=\N$. Consequently,
$q$ will be called the {\em trivializing} pattern.\qed
\end{example}


We say that a pattern $p_1$ {\em induces} another pattern $p_2$ if every numerical semigroup
admitting $p_1$ admits also $p_2$. We say that two patterns are {\em equivalent} if they induce
each other.

\begin{example}
All patterns in Example~\ref{ejemplo 1} are equivalent.
\end{example}

\begin{example}
The trivializing pattern induces the Arf pattern. The Arf pattern and the pattern $2x_1-x_2$ are
equivalent \cite[Proposition 1]{CaFaMu:arf}. Actually, by using the same argument given in that
proposition, it is not hard to prove that for $n\geq 2$, the patterns $x_1+\cdots+x_n-x_{n+1}$
and $x_1+\cdots+x_{n-2}+2x_{n-1}-x_n$ are equivalent. However, in general it is not true that
$x_1+\cdots+ x_n-x_{n+1}$ is equivalent to $x_1+(n-1)x_2-x_3$ (see Example \ref{varios-p}).\qed
\end{example}

\begin{lemma}
\label{lemma:un patro indueix les parts} A pattern $p=\sum_{i=1}^na_ix_i$ induces all patterns
$p_{n'}=\sum_{i=1}^{n'}a_ix_i$ with $n'\leq n$.
\end{lemma}

\begin{proof} Suppose $\Lambda$ admits $p$. Then the integer $p(s_1,\dots,s_n)$ belongs to
$\Lambda$ for every $n$ elements $s_1,\dots, s_n$ in $\Lambda$ with $s_1\geq s_2\geq \dots\geq
s_n$. In particular, we can take $s_{n'+1}=s_{n'+2}=\dots=s_n=0$, and we have that
$p_{n'}(s_1,\dots,s_{n'})$ belongs to $\Lambda$ for every $n'$ elements $s_1,\dots, s_{n'}$ in
$\Lambda$ with $s_1\geq s_2\geq\dots \geq s_{n'}$.
\end{proof}

\begin{lemma} \label{lemma:a un patro li podem intercalar una variable} A pattern
$p=\sum_{i=1}^na_ix_i$ induces all $(n+1)$-length patterns
$$\check{p}_{(j)}=\sum_{i=1}^{j-1}a_ix_i+x_j+\sum_{i=j}^{n}a_ix_{i+1}.$$
\end{lemma}
\begin{proof}
If $\Lambda$ admits $p$, then
$p(s_1,s_2,\dots,s_{j-1},s_{j+1},\dots,s_{n+1})\in\Lambda$ for all $s_1\geq
s_2\geq\dots\geq s_{j-1}\geq s_{j+1}\geq \dots\geq s_{n+1}$. Now,
\begin{multline*}
\check{p}_{(j)}(s_1,s_2,\dots,s_{j-1},s_j,s_{j+1},\dots,s_{n+1})=\\
p(s_1,s_2,\dots,s_{j-1},s_{j+1},\dots,s_{n+1})+s_j,
\end{multline*} which is
clearly in $\Lambda$ for all $s_j\in\Lambda$.
\end{proof}

The next proposition together with Example~\ref{ejemplo 1} points out that
every pattern is either equivalent to the zero pattern or equivalent to a
pattern with the last coefficient negative.
\begin{proposition} \label{proposition:a un patro se li pot treure la cua negativa}
Let $p=\sum_{i=1}^n a_ix_i$ be a pattern. Suppose that $a_{n'}<0$ and that $a_{n'+1}, a_{n'+2},
\dots, a_{n}$ are positive. Then $p$ is equivalent to $p_{n'}=\sum_{i=1}^{n'}a_ix_i$.
\end{proposition}
\begin{proof} The pattern $p$ induces the pattern $p_{n'}$ by Lemma~\ref{lemma:un
patro indueix les parts}. The pattern $p_{n'}$ induces the pattern $p_n$ by
applying  Lemma~\ref{lemma:a un patro li podem intercalar una variable}
several times.
\end{proof}


\section{Admissible patterns}
\label{section:admissible}

For certain patterns $p$ the set $\SS(p)$ is empty and for this reason we are not interested in
them. In this section we characterize those patterns $p$ for which $\SS(p)$ is not empty. To
this end, we need a couple of technical lemmas, one of which will be also used in the last
section.

\begin{lemma}\label{dos} Let $p=\sum_{i=1}^na_ix_i$ be a linear homogeneous polynomial and let
$S=\sum_{i=1}^na_i$. Assume that $\sum_{i=1}^m a_i\geq 0$ for all
$m\in\{1,\ldots,n\}$. Then for all non-negative integers $s_1\geq
s_2\geq\cdots\geq s_n$, $p(s_1,\dots,s_n)\geq S\,s_n$.
\end{lemma}
\begin{proof}
We proceed by induction on $n$. It is clear for $n=1$. Assume that the result holds for any
linear homogeneous polynomial in $n$ unknowns. Let $p=\sum_{i=1}^{n+1}a_ix_i$, let
$S=\sum_{i=1}^{n+1}a_i$ and let $\tilde{S}=\sum_{i=1}^na_i$. Then,
\begin{multline*}p(s_1,\dots,s_{n+1})=p(s_1,\dots,s_{n},0)+a_{n+1}s_{n+1}\\\geq \tilde{S}
s_n+a_{n+1}s_{n+1}\geq \tilde{S} s_{n+1}+a_{n+1}s_{n+1}\geq S s_{n+1} \end{multline*}
\end{proof}

\begin{lemma}
\label{lemma:N admet tots els adimssibles} If a pattern $p=\sum_{i=1}^na_ix_i$ satisfies
$\sum_{i=1}^{n'}a_i\geq 0$ for all $n'\leq n$, then $\N$ admits $p$.
\end{lemma}
\begin{proof} To prove that $\N$ admits $p$ we need to prove that $p(s_1,\dots,s_n)\geq 0$ for all
$s_1\geq s_2\geq\dots \geq s_n$ with $s_1,\ldots,s_n\in \N$. This follows easily by applying
Lemma \ref{dos}, since we know that $p(s_1,\ldots,s_n)\geq (\sum_{j=1}^n a_j)s_n$, which
trivially belongs to $\N$.
%
%
%
\end{proof}

\begin{lemma}\label{11}
 If a pattern $p=\sum_{i=1}^na_ix_i$ does not satisfy $\sum_{i=1}^{n'}a_i\geq 0$ for all
$n'\leq n$, then there is no numerical semigroup admitting $p$.
\end{lemma}
\begin{proof} Suppose that there exists $n'\leq n$ such that $\sum_{i=1}^{n'}a_i< 0$. Let $\Lambda$
be a numerical semigroup and let $l$ be a non-zero element of $\Lambda$. Take
$s_1=s_2=\dots=s_{n'}=l$ and $s_{n'+1}=\dots=s_n=0$. It is obvious that $\sum_{i=1}^{n}a_is_i<
0$ and thus it is not in $\Lambda$.
\end{proof}

As a consequence of the preceeding lemmas we have the next theorem.
\begin{theorem}
\label{theorem:condicions equivalents a admissible} Given a pattern $p=\sum_{i=1}^na_ix_i$, the
following conditions are equivalent.
\begin{itemize}
\item There exists a numerical semigroup that admits $p$, \item $\N$ admits $p$,

\item $\sum_{i=1}^{n'}a_i\geq 0$ for all $n'\leq n$.
\end{itemize}
\end{theorem}

The patterns
satisfying any of the three equivalent conditions in Theorem~\ref{theorem:condicions equivalents
a admissible} will be called {\em admissible} patterns.

\begin{remark} $ $
\begin{itemize}
\item Note that the definition of admissible pattern implies $a_1\geq 0$.

\item  All non-admissible patterns are equivalent.
\end{itemize}
\end{remark}

\section{Strongly admissible patterns}
\label{section:lock-relock}

Given a pattern $p=\sum_{i=1}^na_ix_i$, set
$$p'=
\casos{p-x_1}{if $a_1>1$,}{p(0,x_1,x_2,\dots,x_{n-1})}{otherwise,}$$ and define recursively
$p^{(0)}=p$ and $p^{(i)}=(p^{(i-1)})'$, for $i\in\N\setminus\{0\}$.

A pattern $p$ is said to be {\em strongly admissible} if it is admissible and $p'$ is admissible
as well. We will see that for a strongly admissible pattern $p$, the set $\SS(p)$ is infinite
and that it is possible to check computationally whether or not a numerical semigroup admits
$p$.

\begin{lemma}\label{strictly-admissible}
Let $p$ be a strongly admissible pattern of length $n$. Then for every $k_1\geq \cdots \geq
k_n$, it holds that $p(k_1,\ldots,k_n)\geq k_1\geq \cdots\geq k_n$.
\end{lemma}
\begin{proof}
Assume that $p=\sum_{i=1}^n a_ix_i$. Since $p'$ is admissible, we have that
\[p'(k_1,\ldots,k_n)\geq 0\] ($p'(k_2,\ldots,k_n)\geq 0$, if $a_1=1$), which
leads to \[p(k_1,\ldots,k_n)=k_1+p'(k_1,\ldots,k_n)\geq k_1\]
($p(k_1,\ldots,k_n)=k_1+p'(k_2,\ldots,k_n)\geq k_1$, if $a_1=1$).
\end{proof}

Observe that for a numerical semigroup $\Lambda$ with multiplicity
$m(\Lambda)$, the set $\Lambda\setminus\{\m(\Lambda)\}$ is also a numerical
semigroup. Next corollary shows that this semigroup admits all strongly
admissible patterns admitted by $\Lambda$.

\begin{corollary}\label{a0}
Let $p$ be a strongly admissible pattern and let $\Lambda$ be a numerical semigroup admitting
$p$. Then $\Lambda\setminus\{\m(\Lambda)\}$ also admits $p$.
\end{corollary}
\begin{proof}
Assume that $p$ has length $n$ and let $s_1,\ldots,s_n$ be elements of $\Lambda$ such that
$s_1\geq \cdots\geq s_n$ with $\m(\Lambda)\not\in \{s_1,\ldots,s_n\}$. Then either $s_i=0$ for
all $i$ or $s_1>\m(\Lambda)$. Hence either $p(s_1,\ldots,s_n)=0$ or in view of Lemma
\ref{strictly-admissible}, $p(s_1,\ldots,s_n)\geq s_1>\m(\Lambda)$. In both cases
$p(s_1,\ldots,s_n)\in \Lambda\setminus\{\m(\Lambda)\}$. This proves that
$\Lambda\setminus\{\m(\Lambda)\}\in \SS(p)$.
\end{proof}

This proves that $\SS(p)$ has infinitely many elements if $p$ is a strongly
admissible pattern. We will see in Section~\ref{section:p-sistemas} which
elements we can remove from $\Lambda\in \SS(p)$ so that the resulting
numerical semigroup also admits $p$.


\begin{corollary}
Let $p$ be a strongly admissible pattern of length $n$. Then for any numerical semigroup
$\Lambda$ with conductor $c$, $\Lambda$ admits $p$ if and only if for every $s_1,\dots,
s_n\in\Lambda$ with $c>s_1\geq s_2\geq\dots\geq s_n$, the integer $p(s_1,\dots,s_n)$ belongs to
$\Lambda$.
\end{corollary}
\begin{proof} We need to prove that if $p(s_1,\dots,s_n)$ belongs to $\Lambda$ for all $c>s_1\geq
s_2\geq\dots\geq s_n$, then $p(s_1,\dots,s_n)$ belongs to $\Lambda$ for all $s_1\geq
s_2\geq\dots\geq s_n$ with $s_1\geq c$. Let $s_1\geq \cdots \geq s_n$  be elements in $\Lambda$
such that $s_1\geq c$. In view of Lemma \ref{strictly-admissible}, $p(s_1,\dots,s_n)\geq s_1\geq
c$ and thus $p(s_1,\ldots,s_n)\in \Lambda$.
\end{proof}

This result enables us to check computationally if a strongly admissible
pattern is admitted or not by a numerical semigroup. Observe that for an
admissible pattern $p$ that is not strongly admissible, the best lower bound
we have for $p(s_1,\ldots,s_n)$ is given in Lemma \ref{dos}, which
unfortunately cannot be used to effectively check whether or not a numerical
semigroup admits $p$.

\section{Closures}
\label{section:closures}

A {\em covering} of a numerical semigroup $\Lambda$ with respect to an admissible  pattern $p$
is a numerical semigroup containing $\Lambda$ and admitting $p$. A {\em closure} of a numerical
semigroup $\Lambda$ with respect to an admissible  pattern $p$ (or simply a $p$-closure of
$\Lambda$) is a covering of $\Lambda$ with respect to $p$ not containing properly any other
covering.

If $p$ is an admissible pattern, then $\SS(p)$ is not empty, since by
Theorem~\ref{theorem:condicions equivalents a admissible} $\N$ is in this set. Moreover, notice
that $\N$ is a covering of any numerical semigroup with respect to any admissible pattern.

Let $\Lambda$ be a numerical semigroup. As $\N\setminus \Lambda$ has finitely many elements, we
have that the set of $\{\Gamma \in \SS(p) ~|~ \Lambda\subseteq \Gamma\}$ is finite (and not
empty by the remark made above). Besides, one can easily proof the following result.

\begin{lemma}
Let $p$ be an admissible pattern and let $\Lambda_1,\ldots,\Lambda_n\in \SS(p)$. Then
$\Lambda_1\cap \cdots \cap \Lambda_n\in \SS(p)$.
\end{lemma}

Hence the $p$-closure of $\Lambda$ is $\bigcap_{\Gamma\in \SS(p), \Lambda\subseteq \Gamma}
\Gamma$. However, this construction cannot be (so far) easily performed, since we still do not
have a procedure to construct the set $\SS(p)$. In this section we show how a covering of a
numerical semigroup with respect to certain admissible patterns can be constructed
algorithmically.

\begin{lemma} \label{lemma:p(Lambda) conte 0 i es tancat respecte +} Given any numerical semigroup
$\Lambda$ and any admissible pattern $p=\sum_{i=1}^na_ix_i$ (admitted or not by $\Lambda$), the
set $\{p(s_1,\dots,s_n)\mid s_1\geq s_2\geq \dots \geq s_n\}$ contains $0$ and is closed under
addition.
\end{lemma}
\begin{proof}
This follows easily from the linearity of $p$.
\end{proof}

Given a subset $A$ of $\N$ and an admissible pattern $p$, the set
$\{p(s_1,\dots,s_n)\mid s_1\geq s_2\geq \dots \geq s_n, s_1,\ldots, s_n\in
A\}$ will be denoted by $p(A)$, and the set
$$\underbrace{p(p(p\dots(p}_k(A))\dots))$$
will be denoted by $p^k(A)$.

\begin{remark}
Given a numerical semigroup $\Lambda$ and an admissible pattern $p$, the set $p(\Lambda)$ will
not be a numerical semigroup in general. For instance, if we take $p=2x_1$, then $\N\setminus
p(\Lambda)$ has infinitely many elements and thus $p(\Lambda)$ is not a numerical semigroup.
\end{remark}

\begin{remark} A numerical semigroup $\Lambda$ admits a pattern $p$ if and only if
$p(\Lambda)~\subseteq~\Lambda$.
\end{remark}

We say that a pattern $p=\sum_{i=1}^{n}a_ix_i$ is {\em premonic} if $\sum_{i=1}^{n'}a_i=1$ for
some $n'\leq n$. In particular, all monic patterns are premonic.

\begin{lemma} \label{lemma:si p es premonic alshores p(Lambda) conte Lambda} If $p$ is a premonic
pattern, then $p(\Lambda)$ contains $\Lambda$, for every numerical semigroup~$\Lambda$.
\end{lemma}
\begin{proof}
Suppose that $n'\leq n$ is such that $\sum_{i=1}^{n'}a_i=1$. Let $\Lambda$ be
a numerical semigroup and let $l\in\Lambda$. Then
$$l=\sum_{i=1}^{n'}a_il=p(\overbrace{l,l,\dots,l}^{n'},\overbrace{0,\dots,0}^{n-n'})\in p(\Lambda).$$
\end{proof}

\begin{proposition}
\label{proposition:si p es premonic alshores p(Lambda) es semigrup} If $p$ is a premonic pattern
and if $\Lambda$ is a numerical semigroup, then $p(\Lambda)$ is a numerical semigroup.
\end{proposition}
\begin{proof}
As a consequence of Lemma~\ref{lemma:si p es premonic alshores p(Lambda) conte Lambda} the
number of elements in $\N\setminus p(\Lambda)$ is finite. This, together with
Lemma~\ref{lemma:p(Lambda) conte 0 i es tancat respecte +}, proves that $p(\Lambda)$ is a
numerical semigroup.
\end{proof}

\begin{remark} $ $
\begin{itemize}
\item  A numerical semigroup $\Lambda$ admits a premonic pattern $p$ if and only if
the condition $p(\Lambda)=~\Lambda$ holds.

\item By Proposition~\ref{proposition:si p es premonic alshores p(Lambda) es semigrup}, if $p$
is a premonic pattern and if $\Lambda$ is a numerical semigroup, then the set $p^k(\Lambda)$ is
indeed a numerical semigroup containing~$\Lambda$.
\end{itemize}
\end{remark}

\begin{proposition}
\label{proposition:pk(Lambda)=p(k+1)(Lambda)} Given a numerical semigroup $\Lambda$ and an
admissible premonic pattern $p$, there exists an integer $k$ such that
$p^k(\Lambda)=p^{k+1}(\Lambda)$.
\end{proposition}
\begin{proof} It follows from the inclusion $p^i(\Lambda)\subseteq p^{i+1}(\Lambda)$ and the fact
that there is only a finite number of numerical semigroups containing $\Lambda$.
\end{proof}

\begin{theorem}\label{clausura-premonicos}
Given a numerical semigroup $\Lambda$ and an admissible premonic pattern $p$, there exists a
unique closure of $\Lambda$ with respect to $p$. It is exactly $p^k(\Lambda)$ where $k$ is the
minimal integer such that $p^k(\Lambda)=p^{k+1}(\Lambda)$.
\end{theorem}
\begin{proof} We have to prove that $p^k(\Lambda)$ is a covering of $\Lambda$ with respect to $p$
and that any other covering of $\Lambda$ with respect to $p$ will contain $p^k(\Lambda)$. The
first part is a consequence of the choice of $k$. For the second part, notice that any covering
must contain $p^i(\Lambda)$ for all $i$. In particular, it must contain $p^k(\Lambda)$.
\end{proof}

\section{$p$-systems of generators and $\SS(p)$ in a directed acyclic graph}\label{section:p-sistemas}
In this section we exploit the concept of closure given in the preceding
section in order to introduce the concept of $p$-system of generators for an
admissible pattern $p$. This will enable us to construct recursively the set
$\SS(p)$ and arrange it in a directed acyclic graph.

The idea is the following. Let $\Lambda$ be a numerical semigroup. It is not
hard to prove that given $\lambda\in \Lambda$, the set $\Lambda \setminus
\{\lambda\}$ is a numerical semigroup if and only if $\lambda$ is in the
minimal system of generators of $\Lambda$. Besides, if $\Lambda$ is a
numerical semigroup not equal to $\N$, then so is $\Lambda\cup
\{\g(\Lambda)\}$ (the reader can check that $\Lambda \cup\{n\}$, with $n\in
\N\setminus\Lambda$, is a numerical semigroup if and only if $2n$, $3n$ and
$n+\lambda \in \Lambda$ for all $\lambda \in \Lambda$; see
\cite{oversemigroups}). Note also that if $\lambda$ is a minimal generator of
$\Lambda$ greater than $\g(\Lambda)$, then
$\g(\Lambda\setminus\{\lambda\})=\lambda$, and trivially $\g(\Lambda)$ is a
minimal generator of $\Lambda \cup\{\g(\Lambda)\}$. Thus the operations of
adding the Frobenius number and removing a minimal generator greater than the
Frobenius number are the reverse of one another.

Given a numerical semigroup $\Lambda$, for $n\in \N$, define recursively the semigroup
$\Lambda_n$ as:
\begin{itemize}
\item $\Lambda_0=\Lambda$, \item $\Lambda_{n+1}=\Lambda_n\cup \{\g(\Lambda_n)\}$, if
$\Lambda_n\not=\N$; $\Lambda_{n+1}=\N$, otherwise.
\end{itemize}
Clearly for every numerical semigroup there exists $k\in \N$ such that $\Lambda_k=\N$. Hence
every numerical semigroup can be constructed from $\N$ by removing minimal generators greater
than the Frobenius number of the current numerical semigroup in the chain.

We will do the same for any admissible pattern $p$. First, we need to
introduce the concept of a $p$-system of generators. We will see that minimal
$p$-systems of generators are unique and that $\SS(p)$ is closed under the
operations of adding the Frobenius number and removing $p$-generators greater
than the Frobenius number. This will allow us to construct recursively the set
of all elements of $\SS(p)$.

Let $\Lambda$ be a numerical semigroup and let $p$ be an admissible pattern. As defined above,
we can construct the $p$-closure of $\Lambda$ as the intersection of all numerical semigroups in
$\SS(p)$ containing $\Lambda$ (this intersection is finite, since $\N\setminus \Lambda$ has
finitely many elements). Hence for $\Lambda\in \SS(p)$, we say that $A$ is a $p$-system of
generators of $\Lambda$ if the $p$-closure of $\langle A\rangle$ is equal to $\Lambda$. We will
write $\Lambda=\langle A\rangle_p$, when $A$ is a $p$-system of generators of $\Lambda$.
Clearly, if $A=\{n_1,\ldots,n_r\}$ is a system of generators of $\Lambda$, then $A$ is also a
$p$-system of generators of $\Lambda$. As in \cite{RoGaGaBr}, we show that minimal (with respect
to set inclusion) $p$-systems of generators are unique. The procedure to follow is similar to
the one exposed in the above mentioned paper, and the keystone to generalize it is Lemma
\ref{strictly-admissible}.

As a consequence of Corollary \ref{a0} we obtain the following.
\begin{corollary}\label{a4}
Let $\Lambda$ be in $\SS(p)$, with $p$ a strongly admissible pattern, and let $A$ be a
$p$-system of generators of $\Lambda$. Then $\m(\Lambda)\in A$.
\end{corollary}

\begin{lemma}\label{a5}
Let $p$ be a strongly admissible pattern and let $\Lambda$ be a numerical semigroup. For
$A\subseteq \Lambda$ and $s\in \N$, define
\[A(s)=\{a\in A~|~ a\leq s\}.\]
If $s\in p^n(\langle A\rangle)$, then $s\in p^n(\langle A(s)\rangle)$.
\end{lemma}
\begin{proof}
We proceed by induction on $n$. For $n=0$ the result follows trivially. Assume that the
statement holds for $n$ and let us prove it for $n+1$. Let $s\in p^{n+1}(\langle A\rangle)$.
Then there exist $s_1,\ldots, s_k\in p^n(\langle A\rangle)$ such that $s=p(s_1,\ldots,s_k)$. By
induction hypothesis, for every $i\in \{1,\ldots,k\}$, $s_i\in p^n(\langle A(s_i)\rangle)$. From
Lemma \ref{strictly-admissible}, we deduce that $s\geq s_1\geq \cdots\geq s_k$. Hence
$A(s_k)\subseteq \cdots \subseteq A(s_1)\subseteq A(s)$ and thus $s_1,\ldots,s_k\in p^n(\langle
A(s)\rangle)$. We conclude that $s\in p^{n+1}(\langle A(s)\rangle)$.
\end{proof}

Theorem \ref{clausura-premonicos} and Lemma \ref{a5} allow us to generalize
the proof of \cite[Theorem 6]{RoGaGaBr} to any strongly admissible premonic
pattern.

\begin{theorem}\label{a6}
Let $p$ be a strongly admissible premonic pattern and let $\Lambda\in \SS(p)$. Then $\Lambda$
has a unique minimal $p$-system of generators.
\end{theorem}
\begin{proof}
Assume that $A=\{a_1<a_2<\cdots\}$ and $B=\{b_1<b_2<\cdots\}$ are minimal
$p$-systems of generators, and that $A\not=B$. Let $i=\min\{k~|~
a_k\not=b_k\}$ and suppose without loss of generality that $a_i<b_i$ (this
minimum exists, because $A\not=B$). In view of
Theorem~\ref{clausura-premonicos}, there exists a positive integer $k$ such
that $\Lambda=p^k(\langle A\rangle)=p^k(\langle B\rangle)$. Since $a_i\in
\Lambda=p^k(\langle B\rangle)$, by Lemma \ref{a5}, $a_i\in p^k(\langle
b_1,\ldots,b_{i-1}\rangle)$. However,
$\{b_1,\ldots,b_{i-1}\}=\{a_1,\ldots,a_{i-1}\}$, which implies that $a_i\in
p^k(\langle a_1,\ldots,a_{i-1}\rangle)$, contradicting that $A$ was a minimal
$p$-system of generators of $\Lambda$.
\end{proof}

\begin{example}\label{varios-p}
Let $p$ be a strongly admissible premonic pattern. As we pointed out above, if
$\{n_1,\ldots,n_p\}$ is a minimal system of generators of $\Lambda\in \SS(p)$,
then it is also a $p$-system of generators of $\Lambda$. Thus the cardinality
of a minimal $p$-system of generators is smaller than or equal to that of a
minimal system of generators.
\[\langle 7,15\rangle_{x_1+x_2+x_3-x_4}=\langle
7,15\rangle_{x_1+2x_2-x_3}=\langle 7,15,31,47,48\rangle\]  and  \[\langle
7,15\rangle_{x_1+3x_2-x_3}=\langle 7, 15, 46, 69\rangle.\]

Next example shows that $x_1+x_2+x_3+x_4-x_5$ and $x_1+3x_2-x_3$ are not equivalent.
\[\langle 10,21,23\rangle_{x_1+x_2+x_3+x_4-x_5}=\langle
10,21,23\rangle_{x_1+x_2+2x_3-x_4}= \langle 10,21,23,68\rangle\] and
\[\langle 10,21,23\rangle_{x_1+3x_2-x_3}=\langle 10,21,23,78\rangle \subsetneq \langle
10,21,23,68\rangle.\] \qed
\end{example}

Next we show a procedure to construct the set of all the elements in $\SS(p)$ which is analogous
to the one presented in \cite{RoGaGaBr} for Arf semigroups.

\begin{lemma}\label{a7}
Let $\Lambda$ be in $\SS(p)\setminus\{\N\}$ with $p$ a strongly admissible pattern. Then
$\Lambda\cup \{\g(\Lambda)\}\in \SS(p)$.
\end{lemma}
\begin{proof}
Assume that $p$ has length $n$ and let $s_1,\ldots,s_n$ be elements in
$\Lambda\cup\{\g(\Lambda)\}$ such that $s_1\geq \cdots \geq s_n$. We wonder if
$p(s_1,\ldots,s_n)\in \Lambda\cup\{\g(\Lambda)\}$. We distinguish two cases.
\begin{itemize}
\item If $\g(\Lambda)>s_1$, then $\{s_1,\ldots,s_n\}\subseteq \Lambda$. As $\Lambda\in \SS(p)$,
it follows that $p(s_1,\ldots,s_n)\in \Lambda\subset \Lambda\cup\{\g(\Lambda)\}$.

\item If $\g(\Lambda)\leq s_1$, then by Lemma \ref{strictly-admissible}, $p(s_1,\ldots,s_n)\geq
s_1\geq \g(\Lambda)$ and thus $p(s_1,\ldots,s_n)\in \Lambda\cup \{\g(\Lambda)\}$.
\end{itemize}
\end{proof}

Given a numerical semigroup $\Lambda$, recall that we defined a chain
\[\Lambda=\Lambda_0\subseteq \Lambda_1\subseteq \cdots \subseteq \Lambda_k=\N.\] Note that if
$\Lambda\in \SS(p)$ with $p$ a strongly admissible pattern, then by Lemma
\ref{a7}, the chain $\Lambda=\Lambda_0\subseteq \Lambda_1\subseteq \cdots
\subseteq \Lambda_k=\N$ is a chain of numerical semigroups admitting $p$, and
$\Lambda_i=\Lambda_{i+1}\setminus\{a\}$ for some $a\in \Lambda_{i+1}$. The
following result studies a condition that we must impose on an element $a$ in
a numerical semigroup $\Lambda\in \SS(p)$ for $\Lambda\setminus\{a\}$ to be
again in $\SS(p)$. The proof of this result is analogous to that of of
\cite[Lemma 8]{RoGaGaBr}. We include it here for sake of completeness.

\begin{lemma}\label{a8}
Let $p$ be a strongly admissible premonic pattern, let $\Lambda\in \SS(p)$ and
let $a\in \Lambda$. The following conditions are equivalent:
\begin{itemize}
\item[(1)] $a$ belongs to the minimal $p$-system of generators of $\Lambda$,

\item[(2)] $\Lambda\setminus\{a\}\in \SS(p)$.
\end{itemize}
\end{lemma}
\begin{proof}
Let $A\subseteq \Lambda$ be the minimal $p$-system of generators of $\Lambda$.

Assume that $a\in A$. Then
\[ \Lambda\setminus\{a\}\subseteq \langle \Lambda\setminus\{a\}\rangle_p
    \subseteq \Lambda.\]
>From the uniqueness of $A$ (Theorem \ref{a6}), $\Lambda \not=\langle
\Lambda\setminus\{a\}\rangle_p$. Hence $\langle
\Lambda\setminus\{a\}\rangle_p=\Lambda \setminus \{a\}$. Thus $\Lambda
\setminus\{a\}\in \SS(p)$.

Now assume that $\Lambda\setminus\{a\}\in \SS(p)$. If $a\not\in A$, then
$A\subseteq \Lambda\setminus\{a\}$, which is in $\SS(p)$. Thus $\langle
A\rangle_p\subseteq \Lambda\setminus\{a\}$, contradicting that $\langle
A\rangle_p=\Lambda$.
\end{proof}

The following result (similar to \cite[Proposition 9]{RoGaGaBr} for Arf
numerical semigroups) now can be easily deduced from the observations made so
far and characterizes the leaves in the directed acyclic graph of numerical
semigroups admitting a certain pattern.

\begin{proposition}\label{a9}
Let $p$ be a strongly admissible premonic pattern, and let $\Lambda\in
\SS(p)$. The following conditions are equivalent.
\begin{itemize}
\item[(1)] $\Lambda=\overline{\Lambda}\cup \{\g(\overline{\Lambda})\}$, with
$\overline{\Lambda}\in \SS(p)$.

\item[(2)] The minimal $p$-system of generators of $\Lambda$ contains at least one element
greater than $\g(\Lambda)$.
\end{itemize}
\end{proposition}

\begin{example}
We ``draw'' the set $\SS(p)$ for $p=x_1+x_2+x_3-x_4$. Its associated directed
acyclic graph is given in the figure.

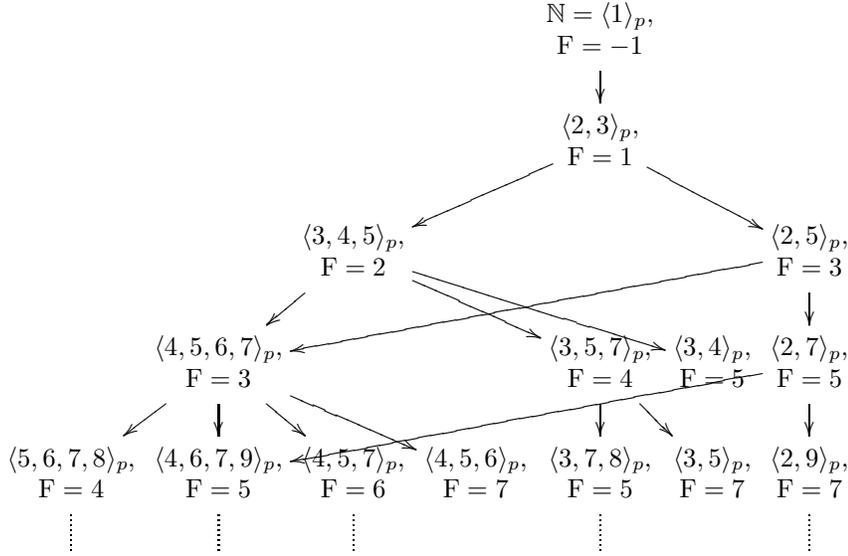
\begin{figure}[h]
\label{f1} \caption{$\SS(p)$, $p=x_1+x_2+x_3-x_4$} \xymatrix @R=1pc @C=.1pc{
  &   &   &   & {\begin{matrix}\N=\langle 1\rangle_p,\\ \g=-1\end{matrix}}\ar[d] &  &    \\
  &   &   &   & {\begin{matrix}\langle 2,3\rangle_p,\\ \g=1\end{matrix}}\ar[dll]\ar[drr]&  & \\
  & & {\begin{matrix}\langle 3,4,5\rangle_p,\\ \g=2\end{matrix}}\ar[dl]\ar[drr]\ar[drrr] & & & &
   {\begin{matrix}\langle 2,5\rangle_p,\\ \g=3\end{matrix}}\ar[d]\ar[dlllll] \\
  & {\begin{matrix}\langle 4,5,6,7\rangle_p,\\ \g=3\end{matrix}} \ar[dl]\ar[d]\ar[dr]\ar[drr] & & &
  {\begin{matrix}\langle 3,5,7\rangle_p,\\ \g=4\end{matrix}}\ar[d]\ar[dr] &
  {\begin{matrix} \langle 3,4\rangle_p,\\ \g=5\end{matrix}} &
  {\begin{matrix}\langle 2,7\rangle_p,\\ \g=5\end{matrix}}\ar[d] \ar[dlllll]\\
  {\begin{matrix}\langle 5,6,7,8\rangle_p,\\ \g=4\end{matrix}}\ar@{..}[d] &
  {\begin{matrix}\langle 4,6,7,9\rangle_p,\\ \g=5\end{matrix}}\ar@{..}[d] &
  {\begin{matrix}\langle 4,5,7 \rangle_p,\\ \g=6\end{matrix}}\ar@{..}[d] &
  {\begin{matrix}\langle 4,5,6\rangle_p,\\ \g=7\end{matrix}} &
  {\begin{matrix}\langle 3,7,8\rangle_p,\\ \g=5\end{matrix}}\ar@{..}[d] &
  {\begin{matrix}\langle 3,5\rangle_p,\\ \g=7\end{matrix}} &
  {\begin{matrix}\langle 2,9\rangle_p,\\ \g=7\end{matrix}}\ar@{..}[d]\\
  & & & & & &
  }
\end{figure}

If we compare it with the directed acyclic graph given in \cite{RoGaGaBr} for
Arf numerical semigroups, one readily sees two main differences. This directed
acyclic graph is not a binary tree; for instance $\langle 4,5,6,7\rangle$ has
four ``sons''. Observe also that the numerical semigroups appearing in the
directed acyclic graph are no longer of maximal embedding dimension, as is the
case for Arf numerical semigroups.

The leaves in the portion of the directed acyclic graph drawn in the figure
are $\langle 3,4\rangle$, $\langle 4,5,6\rangle$ and $\langle 3,5\rangle$.\qed
\end{example}

\section{Substraction patterns}
\label{section:substraction}

The pattern $x_1+x_2+\dots+x_{k}-x_{k+1}$ is called the {\em substraction pattern} of
degree~$k$.

Let $q$ be a rational number. Define $\lceil q\rceil =\min \{z \hbox{ integer} ~|~ q\leq z\}$.

\begin{proposition}
\label{proposition:substraction degree} A semigroup $\Lambda$ with conductor $c$ and
multiplicity $m$ admits the substraction pattern of degree
$\left\lceil\frac{c}{m}\right\rceil+1.$
\end{proposition}
\begin{proof} Set
$$k=\left\lceil\frac{c}{m}\right\rceil+1.$$ Suppose that $s_1\geq \dots\geq s_{k+1}$ belong to $\Lambda$. If
$s_{k-1}=0$ then $s_{k}=s_{k+1}=0$ and it is obvious that $s_1+s_2+\dots+s_{k}-s_{k+1}$ belongs
to $\Lambda$. So, we can assume that $s_{k-1}\geq m$. By the inequality relation between the
$s_i$'s and by the definition of $k$,
\begin{equation}
\label{equation:1} s_1+\dots+s_{k-1}\geq (k-1)m\geq c.
\end{equation}
Besides, by the inequality relation between the $s_i$'s,
\begin{equation}
\label{equation:2} s_{k}-s_{k+1}\geq 0 .
\end{equation}
Now by (\ref{equation:1}) and (\ref{equation:2}), $s_1+s_2+\dots+s_{k}-s_{k+1}$ belongs to
$\Lambda$.
\end{proof}

\begin{remark} \label{remark:inclusions} As a consequence of Lemma~\ref{lemma:a un patro li
podem intercalar una variable}, for each integer $n$, the pattern $x_1+x_2+\dots+x_{n-1}-x_n$
induces the pattern $x_1+x_2+\dots+x_{n}-x_{n+1}$. In particular, by
Proposition~\ref{proposition:substraction degree}, every numerical semigroup admits infinitely
many substraction patterns.
\end{remark}

The {\em substraction degree} of a numerical semigroup is the minimum $k$ such that it admits a
substraction pattern of degree $k$.

The substraction degree of a numerical semigroup gives us an idea of how far from substraction
the numerical semigroup is. It can be thought of as the number of elements that we need to add
in order to be able to substract another element smaller than the first ones.

In particular the substraction degree is always finite and larger than or equal to $1$. It will
be equal to $1$ if and only if the numerical semigroup is $\N$ and it will be $2$ if and only if
the numerical semigroup is Arf and non-trivial.

\begin{remark} \label{remark:substraction degree bound}
Proposition~\ref{proposition:substraction degree} gives the following upper bound for the
substraction degree $s$ of a numerical semigroup with conductor $c$ and multiplicity $m$:
$$s\leq\left\lceil\frac{c}{m}\right\rceil+1.$$
\end{remark}

\begin{example} Let $q$ be a prime power. The Hermitian curve over
${\mathbb F}_{q^2}$ is defined by the affine equation $x^{q+1} = y^q + y$ and it has a single
rational point at infinity. The Weierstrass semigroup at the rational point at infinity is
$\Lambda=\langle q,q+1\rangle$ (for further details see \cite{HoLiPe:agc,Stichtenoth}). Its
multiplicity is $q$ and its conductor is $q(q-1)$. So $\left\lceil\frac{c}{m}\right\rceil+1=q$.
Its substraction degree is $q$. Indeed, by Remark~\ref{remark:substraction degree bound}, it is
enough to prove that the substraction pattern of degree $q-1$ is not admitted (see
Proposition~\ref{tres} for a generalization of this fact). Take $s_1=\cdots=s_{q-1}=q+1$ and
$s_q=q$. Then $s_1+\cdots+s_{q-1}-s_q=(q-1)(q+1)-q=q^2-q-1=c-1\not\in\Lambda$.

This means in particular that the bound in Remark~\ref{remark:substraction degree bound} is
tight.\qed
\end{example}

By using the ideas in this example it is not difficult to prove the following.

\begin{proposition}
\label{leches} Two substraction patterns are equivalent if and only if they have the same
degree.
\end{proposition}

Finally, by Remark~\ref{remark:inclusions}, we can get a graded classification of numerical
semigroups by means of the substraction degree. If we denote
$\SS_i=\SS(x_1+\cdots+x_i-x_{i+1})$, the chain
\[\begin{array}[t]{c}\SS_0 \\ \shortparallel \\ \emptyset\end{array}
\subseteq
\begin{array}[t]{c}\SS_1 \\ \shortparallel \\ \{N\}\end{array}
\begin{array}[t]{ccc}\subseteq&\SS_2&\subseteq \\& \shortparallel& \\ &\begin{array}{c}\text{Arf}\\\text{semigroups}\end{array}&\\\end{array}
\SS_3\ \subseteq\ \cdots\ \subseteq\ \SS_i\ \subseteq\ \SS_{i+1}\subseteq\ \cdots\] contains all
numerical semigroups and it is non-stabilizing.

 Next we give a lower bound for the substraction degree based on the structure of the
Ap\'ery set of the numerical semigroup.

Let $\Lambda$ be a numerical semigroup and $\lambda\in \Lambda\setminus\{0\}$. The Ap\'ery set
(see \cite{apery}) of $\lambda$ in $\Lambda$ is the set
\[ \Ap(\Lambda,\lambda)= \{ \gamma\in \Lambda ~|~ \gamma-\lambda\not\in \Lambda\}.\]
It can be easily shown that given $i\in \{0,\ldots,\lambda-1\}$, if $w(i)$ is the least element
in $\Lambda$ congruent with $i$ modulo $\lambda$, then
$\Ap(\Lambda,\lambda)=\{w(0)=0,w(1),\ldots,w(\lambda-1)\}$ and thus this set has finitely many
elements.

Given $\lambda,\lambda'\in \Lambda$, we write $\lambda\leq_\Lambda \lambda'$
if there exists $\lambda''\in \Lambda$ such that $\lambda'=\lambda+\lambda''$
($\lambda<_\Lambda \lambda'$ denotes $\lambda\leq_\Lambda \lambda'$ and
$\lambda\not=\lambda'$). If $w$ and $w'$ are elements of
$\Ap(\Lambda,\lambda)$ such that $w-w'\in \Lambda$, then clearly $w-w'\in
\Ap(\Lambda,\lambda)$. Thus in some way the partial order $\lambda\leq_\Lambda
\lambda'$ can be restricted to the set $\Ap(\Lambda,n)$. A {\em chain} in
$\Ap(\Lambda,\lambda)$ is a sequence of the form $w_1<_\Lambda \cdots
<_\Lambda w_d$, and we say that $d$ is the {\em length} of the chain. We
define the {\em Ap\'ery depth} of $\Lambda$ as the maximum length of the
chains in $\Ap(\Lambda,\m(\Lambda))$. As the cardinality of
$\Ap(\Lambda,\m(\Lambda))$ is $\m(\Lambda)$, the Ap\'ery depth of $\Lambda$ is
bounded by $\m(\Lambda)$.

\begin{example}
Let $\Lambda$ be a numerical semigroup of maximal embedding dimension, that is to say, a
numerical semigroup minimally generated by $\{m=n_1<n_2<\cdots<n_m\}$. Then the reader can
easily check that $\Ap(\Lambda,m)=\{0,n_2,\ldots,n_m\}$ and thus the Ap\'ery depth of $\Lambda$
is 2.

Now let $\Lambda=\langle m,n\rangle$, with $m<n$ and $\gcd\{m,n\}=1$. Then
$\Ap(\Lambda,m)=\{0,n, 2n,\ldots, (m-1)n\}$ and the Ap\'ery depth of $\Lambda$ is $m$.\qed
\end{example}

The Ap\'ery depth yields a lower bound on the substraction degree as we see
next.

\begin{proposition}\label{d1}
Let $\Lambda$ be a numerical semigroup with Ap\'ery depth $d$ and substraction degree $s$. Then
$d\leq s$.
\end{proposition}
\begin{proof}
Let $w_1<_\Lambda \ldots <_\Lambda w_d$ be a chain of maximal length (this implies that $w_1=0$)
in $\Ap(\Lambda,\m(\Lambda))$. Then $(w_2-w_1)+(w_3-w_2)+ \cdots +(w_d-w_{d-1})= w_d-w_1=w_d$.
Let $x_i=(w_{i+1}-w_i)$ for $i\in \{1,\ldots, d-1\}$ and let $x_d=\m(\Lambda)$. Then
$x_1+\cdots+x_{d-1}-x_d= w_d-\m(\Lambda)\not \in \Lambda$. As $x_i\not=0$ for all $i\in
\{1,\ldots,d-1\}$, this in particular implies that $x_i\geq x_d=\m(\Lambda)$ for all $i\in
\{1,\ldots,d-1\}$. This shows that $\Lambda$ does not admit the pattern
$x_1+\cdots+x_{d-1}-x_d$, which implies that $d-1<s$. Hence $d\leq s$.
\end{proof}

Unfortunately the other inequality (and thus the equality) does not hold.

\begin{example}
Let $\Lambda=\langle 3,8,13\rangle$. The reader can check that $\Lambda\in
\SS(x_1+x_2+x_3-x_4)$. Observe that $8+8-6=10\not\in \Lambda$, which in particular implies that
$\Lambda\not\in \SS(x_1+x_2-x_3)$. Thus the substraction degree of $\Lambda$ is $3$ and its
Ap\'ery depth is $2$ ($\Lambda$ has maximal embedding dimension).\qed
\end{example}

\section{Boolean patterns}
\label{section:boolean}

A pattern is called {\em boolean} if all its coefficients are either $1$ or $-1$. Notice that
the Arf pattern as well as all substraction patterns are boolean.

Let $p$ be the substraction pattern of degree $k$. Observe that $p^{(k)}=-x_1$ is not admissible
whereas for $i<k$, $p^{(i)}$ is an admissible pattern. Generalizing this idea we define the {\em
admissibility degree} of a pattern $p$ as the least $k$ such that $p^{(k)}$ is not admissible.
If this minimum does not exist (this occurs exactly for those patterns described in
Example~\ref{ejemplo 1}), then the admissibility degree is said to be $\infty$. Clearly if a
pattern $p$ is not admissible, then its admissibility degree is $0$.

\begin{lemma}\label{uno}
A boolean pattern $p$ with finite positive admissibility degree $k$ can be written as
$$p(x_1,\dots,x_n)=f(x_1,\dots,x_{k-1})+g(x_{k},\dots,x_{l})+h(x_{l+1},\dots,x_n),$$
where all coefficients in $f$ are positive, both $g$ and $h$ are admissible, the sum of all
coefficients of $g$ is equal to $0$ and the sum of all coefficients of $h$ is positive.
\end{lemma}

\begin{proof}
Assume that $p=\sum_{i=1}^na_ix_i$ with $a_i\in\{-1,1\}$. By hypothesis $p$ can be expressed as
$x_1+\cdots+x_k+q(x_{k+1},\dots,x_n)$ where $q$ is a non-admissible pattern such that
$x_1+q(x_2,\dots,x_{n-k+1})$ is admissible. By Theorem~\ref{theorem:condicions equivalents a
admissible}, this means that there exists $l> k$ such that $\sum_{i=k+1}^la_i=-1$. Taking the
largest of such integers, we obtain that $\sum_{i=l+1}^ma_i>0$ for $m\in\{l+1,\ldots,n\}$. The
result follows by taking $f(x_1,\dots,x_{k-1})=x_1+\cdots+x_{k-1}$,
$g(x_{k},\dots,x_{l})=x_k+\sum_{i=k+1}^la_ix_i$ and $h(x_{l+1},\dots,x_n)=\sum_{i=l+1}^na_ix_i$.
\end{proof}

The next property stresses how the concept of admissibility degree generalizes
that of the degree of a substraction pattern.

\begin{proposition}\label{induce}
A boolean pattern with admissibility degree $k$ induces the substraction pattern of degree $k$.
\end{proposition}
\begin{proof}
Let $p$ be a boolean pattern with admissibility degree $k$. In view of Lemma \ref{uno}, $p$ can
be expressed as $p=f+g+h$. Assume that $\Lambda$ admits $p$ and let us prove that $\Lambda$ also
admits $x_1+\cdots+x_k-x_{k+1}$. Let $s_1\geq \cdots \geq s_k\geq s_{k+1}$ be elements of
$\Lambda$. From the proof of Lemma \ref{uno}, one easily deduces that
$g(s_k,s_{k+1},\ldots,s_{k+1})=s_k-s_{k+1}$. Hence
\begin{multline*}
p(s_1,\ldots,s_k,s_{k+1},\ldots,s_{k+1},0,\ldots,0)=\\
f(s_1,\ldots,s_{k-1})+g(s_k,s_{k+1},\ldots,s_{k+1})+h(0,\ldots,0)\\=s_1+\cdots+s_k-s_{k+1}\in
\Lambda.
\end{multline*}
\end{proof}

Hermitian numerical semigroups can be used to discriminate patterns with
different admissible degrees, as we see next.

\begin{proposition}\label{tres}
The numerical semigroup $\langle q,q+1\rangle$, with $q\geq 2$ admits a
boolean pattern if and only if its admissibility degree is greater than or
equal to $q$.
\end{proposition}
\begin{proof}
Let $p$ be a boolean pattern of length $n$ and admissibility degree $k$. Let
$f$, $g$ and $h$ be as in Lemma \ref{uno}. The sum of the coefficients of $g$
is $0$ and the sum of the coefficients of $h$ is a non-negative (in fact
positive) integer amount, say $S$.

Assume that $k\geq q$ and let $s_1\geq \cdots \geq s_n$ be elements of $\langle q,q+1\rangle$.
We must prove that $p(s_1,\ldots,s_n)\in \langle q,q+1\rangle$. We distinguish two cases:
\begin{itemize}
\item If $s_k<q$, then $s_k=0=s_{k+1}=\cdots=s_n$. Hence
$p(s_1,\ldots,s_n)=f(s_1,\ldots,s_{k-1})=s_1+\cdots+s_{k-1}$ which trivially belongs to $\langle
q,q+1\rangle$.

\item If $s_k\geq q$, then $p(s_1,\ldots,s_n)=f(s_1,\ldots,s_{k-1})+g(s_{k},\ldots,s_l)+
h(s_{l+1},\ldots,s_n)$. By Lemmas \ref{dos} and \ref{11}, we deduce that
$g(s_{k},\ldots,s_l)\geq 0$ and $h(s_{l+1},\ldots,s_n)\geq S s_n\geq 0$. Hence
$p(s_1,\ldots,s_n)\geq (k-1)s_{k-1}\geq (k-1)q$, since $s_{k-1}\geq s_k\geq
q$. As we are assuming that $k\geq q$, we obtain that $p(s_1,\ldots,s_n)\geq
(q-1)q$, which is the conductor of $\langle q,q+1\rangle$. This implies that
$p(s_1,\ldots,s_n)\in \langle q,q+1\rangle$.
\end{itemize}


Assume now that $k<q$ and that $\langle q,q+1\rangle$ admits $p$. By Lemma \ref{induce}, the
semigroup $\langle q,q+1\rangle$ also admits $x_1+\cdots+x_k-x_{k+1}$. Then, by evaluating this
pattern in $s_1=\cdots=s_k=q+1$ and $s_{k+1}=q$, one gets that $k(q+1)-q$ should be in $\langle
q,q+1\rangle$. However, $\Ap(\langle q,q+1\rangle,q)=\{0,q+1,2(q+1), \ldots,(q-1)(q+1)\}$, and
thus $k(q+1)\in\Ap(\langle q,q+1\rangle,q)$, which means that $k(q+1)-q\not\in \langle
q,q+1\rangle$, a contradiction.
\end{proof}

This result has a nice consequence.

\begin{corollary}\label{cuatro}
Two equivalent boolean patterns have the same admissibility degree.
\end{corollary}

\begin{example}
Let $p(x_1,x_2)=10x_1-7x_2$. Note that $p^{(3)}=7x_1-7x_2$, which is admissible. Nevertheless
$p^{(4)}(x_1,x_2)=6x_1-7x_2$ which is not admissible. However $\langle 5,6\rangle$ admits $p$.
This example points out that Proposition \ref{tres} could be false for non-boolean patterns.
\qed
\end{example}

\begin{lemma}\label{nuevo}
  Let $\Lambda$ be an Arf numerical semigroup. Take $s_1\geq \cdots \geq s_n\in \Lambda$. Then
  $s_1+\sum_{i=2}^n a_i s_i\in \Lambda$ for any $\{a_2,\ldots,a_n\}\subseteq \{-1,1\}$ such that
  $\sum_{i=2}^m a_i\geq 0$ for all $m\in \{2,\ldots,n\}$.
\end{lemma}
\begin{proof}
    Let $x=s_1+\sum_{i=2}^n a_i s_i$. We use induction on $n$. For $n\in\{2,3\}$, the result
    follows trivially from the definition of Arf numerical semigroup. Assume that $n>3$. If
    $a_i=1$ for all $i$, then we are done. Thus assume on the contrary that $a_i=-1$ for some
    $i\in \{2,\ldots,n\}$, and let $i$ be the minimum integer fulfilling this condition. From
    the hypothesis, we deduce that $i>2$. Then $x=s_1+s_2+\cdots+s_{i-1}-s_i+\sum_{j=i+1}^n
    a_js_j$. Let $s_1'=s_1+s_{i-1}-s_i$. As $\Lambda$ is Arf, $s_1'\in \Lambda$. Then
    $x=s_1'+s_2+\cdots+s_{i-2}+\sum_{j=i+1}^n a_j s_j$, which is an expression of $x$ with
    length less than $n$ and fulfilling the hypothesis of the statement. By the induction hypothesis
    we deduce that $x\in \Lambda$.
\end{proof}

\begin{proposition}
\label{cinco}
\begin{enumerate}
\item
All boolean patterns with admissibility degree $0$ are equivalent.
\item
All boolean patterns with admissibility degree $1$ are equivalent to the trivializing pattern.
\item
All boolean patterns with admissibility degree $2$ are equivalent to the Arf pattern.
\end{enumerate}
\end{proposition}

\begin{proof}
The first point is trivial since the patterns with admissibility degree $0$
are admitted by no semigroup. By Proposition~\ref{induce}, it is enough to
prove that the trivializing pattern induces any pattern with admissibility
degree $1$ and that the Arf pattern induces any pattern with admissibility
degree $2$. The first part is to say that any pattern with admissibility
degree $1$ is admitted by $\N$, which is obvious. The second part follows
easily from Lemma \ref{nuevo}.
\end{proof}

\begin{example}
Again, Proposition~\ref{cinco} could be false for non-boolean patterns. For
instance, the pattern $5x_1-5x_2$ has admissibility degree $1$ as does the
trivializing pattern. However, the trivializing pattern is admitted only by
$\N$, while $5x_1-5x_2$ is admitted by any numerical semigroup containing $5$
and not necessarily the trivial semigroup. On the other hand, the pattern
$10x_1-9x_2$ has admissibility degree $2$ as does the Arf pattern. However,
the semigroup $$\Lambda=\langle 4,5,11\rangle=\{0,4,5,8,\dots\}$$ which is
obviously not Arf since $5+5-4=6\not\in\Lambda$, admits $10x_1-9x_2$. \qed
\end{example}

Unfortunately, we can not get the converse of Corollary~\ref{cuatro} for
admissibility degree greater than $2$ as we did in Proposition~\ref{leches}
for substraction patterns and in Proposition~\ref{cinco} for patterns with
admissibility degree less than or equal to $2$.

\begin{example}
\label{ejemplillo}
There exist boolean patterns with the same admissibility degree that are not
equivalent. For instance, the semigroup
$$\langle 5,6,13\rangle= \{0, 5, 6, 10, 11, 12, 13, 15, \dots\}$$
admits the pattern
$$p_1=x_1+x_2+x_3-x_4$$
but it does not admit the pattern
$$p_2=x_1+x_2+x_3+x_4-x_5-x_6$$
($\langle 5,6,13\rangle_{p_2}=\langle {5, 6, 13, 14}\rangle$) and they both have admissibility
degree $3$. \qed
\end{example}

Proposition~\ref{tres} can be extended in order to prove that for $k>2$ there exist infinitely
many boolean patterns with admissibility degree $k$ that are not equivalent.

\begin{proposition}
For $k>2$, the semigroup
$$\Lambda=\langle q,q+1\rangle\cup\{(k-1)(q+1)+1,\ldots,(k-1)(q+1)+(q-k-1)\}\cup\{i\in\N\mid
i\geq kq\}$$ admits a boolean pattern of admissibility degree $k$,
$$p=\sum_{i=1}^l a_ix_i=f(x_1,\ldots,x_{k-1})+g(x_k,\ldots,x_l)+h(x_{l+1},\ldots,x_n)$$
(where $f$, $g$, $h$ are as in Lemma~\ref{uno}) if and only if $d\leq q-k-1$, with
$d=\max_j\left(\sum_{i=k}^ja_i\right)$.
\end{proposition}
\begin{proof}
First note that $\Lambda$ is a semigroup because $(k-1)(q-1)+q=kq+k-1\geq kq$.
Suppose $d\leq q-k-1$ and $s_1\geq s_2\geq \dots\geq s_n$. Let $I$ be such
that $s_I\neq 0$ and $s_{I+1}=0$. We can assume that $I>k$, because if $I\leq
k$ it is clear that $p(s_1,\dots,s_n)\in\Lambda$. We can also assume that
$s_1\leq q+1$ because otherwise $s_1\geq 2q$ and $p(s_1,\dots,s_n)\geq
f(s_1,\dots,s_{k-1})\geq kq$. Now let $J$ be such that $s_J=q+1$ and
$s_{J+1}<q+1$ (if such a $J$ does not exist, clearly
$p(s_1,\ldots,s_n)\in\Lambda$). If $J<k$, then $p(s_1,\dots,s_n)=J(q+1)+aq$
for some $a\geq 0$. Besides, if $J>l$, then $p(s_1,\dots,s_n)\geq kq$. So we
can assume $k\leq J\leq l$.
 In this case,
 \begin{multline*}p(s_1,\dots,s_n)=(k-1)(q+1)+\sum_{i=k}^Ja_i(q+1)+\sum_{i=J+1}^Ia_iq\\=
(k-1)(q+1)+q\sum_{i=k}^Ia_i+\sum_{i=k}^Ja_i\\  \left\{\begin{array}{ll}
\geq kq & \text{\ if\ } I>l \text{\ or\ } \sum_{i=k}^Ia_i>0, \\
=(k-1)(q+1)+d'\text{\ with\ }d'\leq q-k-1 & \text{\ if\ } I\leq l \text{\ and\ }
\sum_{i=k}^Ia_i=0.
\end{array} \right.
\end{multline*}
For the converse, if $d\geq q-k$, let $J$ be such that $q-k=\sum_{i=k}^J a_i$, then
\begin{multline*}
p(\underbrace{q+1,\dots,q+1}_{J)},\underbrace{q,\dots,q}_{l-J)},0,\dots,0)=
(k-1)(q+1)+(q-k)(q+1)-(q-k)q\\=(k-1)(q+1)+q-k=qk-1\not\in\Lambda
\end{multline*}
\end{proof}

This proves that two equivalent patterns of the same admissibility degree $k>2$ must have the
same value $d=\max_j\left(\sum_{i=k}^ja_i\right)$.

\begin{example}
Let  $p_1$ and $p_2$ be the patterns defined in Example~\ref{ejemplillo}. According to Lemma
\ref{uno}, the patterns $p_1$ and $p_2$ can be expressed as $p_1=f_1+g_1+h_1$ and
$p_2=f_2+g_2+h_2$ with $p_1=x_1+x_2=p_2$, $g_1=x_3-x_4$, $g_2=x_3+x_4-x_5-x_6$ and $h_1=0=h_2$.
Hence, the value of $d$ for $p_1$ and $p_2$ is $1$ and $2$, respectively. \qed
\end{example}

\begin{example}
Furthermore, there exist patterns with the same admissibility degree $k$ and the same $d$
which
are not equivalent. For instance, the semigroup
$$\langle 7,8,17,26\rangle= \{0, 7, 8, 14, 15, 16, 17, 21, 22, 23, 24, 25, 26, 28, \dots\}$$
admits the pattern
$$p_1=x_1+x_2+x_3-x_4$$
but it does not admit the pattern
\[p_2=x_1 + x_2 + x_3 - x_4 + x_5 + x_6 - x_7 - x_8\]
($\langle 7,8,17,26\rangle_{p_2}=\langle 7, 8, 17, 18, 27\rangle$) and they both have
admissibility degree $3$ and $d=1$.

\qed
\end{example}


\end{document}